\documentclass[12pt]{amsart}
\usepackage{amssymb,amscd,amsmath}
\usepackage[all]{xy}
\newtheorem{thm}{Theorem}[section]
\newtheorem{theorem}[thm]{Theorem}

\newtheorem{prop}[thm]{Proposition}
\newtheorem{lemma}[thm]{Lemma}

\newtheorem{rem}[thm]{Remark}

\newtheorem{lem}[thm]{Lemma}

\theoremstyle{remark}

\theoremstyle{definition}

\newcommand{\A}{{\mathsf{A}}}
\newcommand{\B}{{\mathsf{B}}}
\newcommand{\CC}{{\mathsf{C}}}
\newcommand{\D}{{\mathsf{D}}}
\newcommand{\E}{{\mathsf{E}}}
\newcommand{\FF}{{\mathsf{F}}}
\newcommand{\GG}{{\mathsf{G}}}
\renewcommand{\H}{{\mathsf{H}}}
\newcommand{\I}{{\mathsf{I}}}
\newcommand{\X}{\mathbf{X}}

\newcommand{\R}{{\mathbb{R}}}
\newcommand{\C}{{\mathbb{C}}}

\newcommand{\Z}{{\mathbb{Z}}}

\newcommand{\Tor}{\mathrm{Tor}}

\newcommand{\GL}{\mathrm{GL}}

\newcommand{\Span}{\mathrm{Span}}

\newcommand{\sgn}{\mathrm{sgn}}

\newcommand{\id}{\mathrm{id}}

\newcommand{\m}{\mathfrak m}
\newcommand{\M}{\mathfrak m}

\newcommand{\Alt}{{\raise 2pt\hbox{$\scriptstyle\bigwedge$}}}

\begin{document}

\title
{Coxeter Cochain Complexes}

\author{Michael Larsen}
\email{mjlarsen@indiana.edu}
\address{Department of Mathematics\\
    Indiana University \\
    Bloomington, IN 47405\\
    U.S.A.}

\author{Ayelet Lindenstrauss}
\email{alindens@indiana.edu}
\address{Department of Mathematics\\
    Indiana University \\
    Bloomington, IN 47405\\
    U.S.A.}

\thanks{Michael Larsen was partially supported by NSF grant DMS-1101424.}

\begin{abstract}
We define the \emph{Coxeter cochain complex} of a Coxeter group $(G,S)$ with coefficients in a $\Z[G]$-module $A$.  This is closely related to the complex of simplicial cochains on the abstract simplicial complex $I(S)$ of the commuting subsets of $S$.  We give some representative computations of Coxeter cohomology
and explain the connection between the Coxeter cohomology for groups of type $\A$, 
the (singular) homology of certain configuration spaces, and the (Tor) homology of
certain local Artin rings.
\end{abstract}

\maketitle

\newpage
\section{Introduction}
\newcommand{\HH}{\mathrm{HH}}
\newcommand{\Sym}{\mathrm{Sym}}

Let $\Gamma$ be a graph on a vertex set $S$ and $I(S)$ the set of independent subsets of
$S$.
The set $I(S)$ is an abstract simplicial complex, known as the independence complex of $S$.
The (co)homology of the geometric realization $|I(S)|$ is a non-trivial natural invariant of $\Gamma$.  
There is a substantial literature concerning the topology of $|I(S)|$.
See, e.g., \cite{AB,BLN,EH,E,J,K} and the references contained therein.

Let $\Gamma$ be the graph underlying a Coxeter group
$G$ and $A$ be any $\Z[G]$-module.  The purpose of this paper is to 
introduce a natural cochain complex which can be thought of, very roughly, as a version of the cohomology
of $|I(S)|$ with coefficients in $A$.  If $A$ is a trivial $\Z[G]$-module, our cochain complex is isomorphic to a shifted version of the reduced simplicial cochain complex of    $|I(S)|$ with coefficients in the abelian group $A$, but if the action is nontrivial, it is something new.
The cohomology of our complex,
the \emph{Coxeter cohomology of $G$ with coefficients in $A$}, is in many cases explicitly computable. 
We present some simple calculations for finite Coxeter groups.
Coxeter cohomology with non-trivial coefficients 
is also related to the topology of certain configuration spaces and the homological algebra of algebras of
order $3$ fatpoints.

Indeed, Coxeter cohomology of groups of type $\A$ appears implicitly in work of Peeva, Reiner, and Welker \cite{PRW}.
Let $\X_{n,k}$ denote the subset of $\R^n$
consisting of $n$-tuples such that no $k$  of the coordinates are equal,
and let $R_{m,k} =  \C[x_1,\ldots,x_m]/\m^k$, where
$\m = (x_1,\ldots,x_m)$.
In \cite{PRW}, the (singular) homology of $\X_{n,k}$, regarded as a representation of $\A_{n-1}\cong S_n$,
is related to the (Tor) homology of $R_{m,k}$, regarded as a representation of $\GL_m(\C)$.  For $k=2$, this correspondence is almost obvious.  For $k=3$,
it is mediated by Coxeter cohomology.
Explicitly, there are isomorphisms
$$H_k(\X_{n,3}; \C)\cong H^{n-k}_C(\A_{n-1}, \C[\A_{n-1}])\otimes\sgn$$
as $\A_{n-1}$-modules and
$$\Tor_i^{R_{m,3}}(\C,\C)\cong \bigoplus_{0\leq j\leq i} H^j_C(\A_{i+j-1} , V^{\otimes (i+j)})$$
as $\GL_m(\C)$ representations, where $V\cong \C^m$ is the standard representation of $\GL_m(\C)$ on $\C^m$.

We remark that the original motivation for this paper was to understand the $k=3$ case of the result \cite{L} comparing Hochschild homology of the rings
$R_{m,k}$ to the homology of $(S^1)^n$ relative to the subspace of $n$-tuples for  which there is a $k$-fold collision.
We believe that it should be possible to relate both sides of this comparison theorem to the Coxeter cohomology of
affine Weyl groups of type $\A$ and hope to discuss this in a future paper.

We would like to thank Vic Reiner and Volkmar Welker for their pointers to relevant literature.

\smallskip
This paper is dedicated with love to the memory of Joram Lindenstrauss.

\section{Coxeter cochain complexes}
Let $G$ be a Coxeter group  with a set $S$ of generators.  Let $I(S)$ denote the set of all subsets of mutually
commuting elements of $S$ or, equivalently,
the set of independent subsets of the vertices of the Coxeter graph of $G$.

Let $A$ be a  $\Z[G]$-module and $<$ an ordering of $S$.
If $T\subset S$ and $s\in S\setminus T$ such that
$T\cup \{s\}\in I(S)$, we define
$$d_{T, s}\colon A^{\langle T\rangle}\to A^{\langle T\cup \{s\}\rangle}$$
by
$$d_{T,s}(v) = (-1)^{|\{t\in T\mid t < s\}|}(v+s(v)).$$
Thus, if $T\cup \{s,s'\}\in I(S)$ (where $s$ and $s'$ are distinct   elements of $S\setminus T$),
then $s$ and $s'$ commute, so we have
\begin{equation}
\label{square}
d_{T\cup\{s\},s'}\circ d_{T,s}(a) + d_{T\cup\{s'\},s}\circ d_{T,s'}(a) = 0.
\end{equation}

We define
the \emph{Coxeter cochain complex} of $G$ with respect to $A$ and $<$ 
to be the cochain complex $X_C^\cdot := X_C^\cdot(G,S,<,A)$  where
$$X_C^k = \bigoplus_{\{T\in I(S): |T| = k\}} A^{\langle T\rangle},$$
and $d^k\colon X_C^k\to X_C^{k+1}$ is given by $\sum_{T,s} d_{T,s}$
where the sum is taken over all $k$-element subsets $T$ of $S$ and for every such $T$, over all $s\not\in T$
for which $T\cup\{s\}\in I(S)$.  The fact that $d^{k+1}\circ d^k = 0$ follows from equation (\ref{square}).

If $<_1$ and $<_2$ are both orderings of $S$, then the complexes 
$$X_{i,C}^\cdot := X_C^\cdot(G,S,<_i,A)$$
for $i=1,2$ are isomorphic to one another.  Indeed, we map
the summand associated to $T\in I(S)$ in $X_{1,C}^\cdot $ to the summand associated to $T$ in $X_{2,C}^\cdot$ via
multiplication by  a sign $\epsilon_T$, which is the sign of the permutation of $T$ which, applied to the ordering of
$T$ by $<_1$, produces the ordering of $T$ by $<_2$.  For this reason, for the remainder of the paper, the choice
ordering of $S$ will be ignored.  We understand a Coxeter group $G$ to be endowed with a fixed set $S$ of generators,
which are henceforth omitted from our notation.  The cohomology of $X_C^\cdot (G,A)$ will be denoted 
$H^i_C(G,A)$.  We often use Killing-Cartan notation for $G$, for instance, writing $\A_n$ for $S_{n+1}$.

\begin{rem}
 \label{geometric}
When $A$ is a trivial $\Z[1/2][G]$-module, 
the Coxeter cochain complex $X^\cdot_C(G,A)$ is isomorphic to a shift of the reduced complex of the simplicial cochains on the independence complex $I(S)$ with coefficients in $A$ as an abelian group, and so
$$H^i_C(G,A)\cong\tilde H^{i-1} (|I(S)|; A)$$
for all $i\geq 0$.
\end{rem}
\begin{proof}
In the Coxeter cochain complex. each  $X^k_C(G,A)$ is spanned over $A$ by independent sets $T\subseteq S$ consisting of $k$ elements, just as the $(k-1)$-dimensional simplicial cochains on $|I(S)|$ are; $X^0_C(G,A)\cong A$ including into   $X^1_C(G,A)$ when $S\neq \emptyset$ is the reduction.   The boundary maps in both cases agree except for multiples of $2$, which is invertible in $A$, so one can map the Coxeter cochain complex to the simplicial cochain complex dividing by $2^k$ in each degree $k$ to get the desired chain isomorphism.
\end{proof}

\begin{lem}
\label{ses}
If  $0\to A_1\to A_2\to A_3\to 0$ is a short exact sequence of $\Z[1/2][G]$-modules, then
there is a long exact sequence
$$0\to H^0_C(G,A_1)\to H^0_C(G,A_2)\to H^0_C(G,A_3)\to H^1_C(G,A_1)\to \cdots.$$
Moreover, if the sequence is split, we have $H^i(G,A_2)\cong H^i(G, A_1)\oplus H^i(G, A_3)$ for all $i$.
\end{lem}

\begin{proof}
The long exact sequence comes from the (vertical) exact sequence of (horizontal) complexes
$$0\to X_C^\cdot(G,A_1) \to X_C^\cdot(G,A_2) \to X_C^\cdot(G,A_3)\to 0.$$
The exactness of the columns follows from the exactness of the functor $A\mapsto A^{\langle T\rangle}$,
which, in turn, depends on the fact that the order of $\langle T\rangle$ is invertible in $\Z[1/2]$.
If the original short exact sequence of modules is split, the isomorphism is trivial (and holds for general $\Z[G]$-modules).
\end{proof}

Henceforth, we restrict attention to the case of modules $A=V$ which are vector spaces over a field $K$.
\begin{lem}
\label{kunneth}
Let $V_1, V_2,\ldots, V_n$ be representations of Coxeter groups $G_1,G_2,\ldots,$ $ G_n$, and let
$V = V_1\boxtimes\cdots\boxtimes V_n$ and $G = G_1\times \cdots \times G_n$.  Then
$$H^i_C(G,V) \cong  \bigoplus_{i_1+\cdots+i_n = i} H^{i_1}_C(G_1,V_1)\otimes\cdots\otimes H^{i_n}_C(G_n,V_n).$$
\end{lem}

\begin{proof}
We understand $G$ to be endowed with a set of generators $S = S_1\cup\cdots\cup S_n$, where $S_i$ is the set of
generators of $G_i$.  Then the Coxeter cochain complex of $(G,S)$ with respect to $V$ is isomorphic to the tensor power of the complexes of $(G_i,S_i)$ with respect to $V_i$.  The lemma follows immediately from the K\"unneth formula.
\end{proof}

Given a Coxeter group $G$ with graph $\Gamma$ and $s\in S$, we define $G_s := \langle S\setminus \{s\}\rangle$ 
and $G^s := \langle S\setminus B_s(1)\rangle$, where $B_s(1)$ denotes the set of elements of $S$ within distance $1$
of $s$ on $\Gamma$, i.e., the set consisting of $s$ itself and all its adjacent vertices.  Note that every element
of $G^s$ commutes with $s$, so the invariant space
$V^s$ is a $G^s$-representation.

\begin{prop}
\label{les}
For every  Coxeter group $G$, every representation $V$ of $G$ of characteristic $\neq 2$, and every $s\in S$, there exists
a long exact sequence
$$\cdots \to H_C^{i-1}(G^s,V^s)\to H_C^i(G,V)\to H_C^i(G_s,V)\to H_C^i(G^s,V^s)\to \cdots.$$
\end{prop}

\begin{proof}
There is a short exact sequence of complexes
$$0\to X_C^{*-1} (G^s,V^s) \to
 X_C^*(G,V) \to 
 X_C^*(G_s,V)\to 0$$
which comes from 
\begin{multline*}
0\to \!\!\!\!\!\!  \bigoplus_{\{T\in I(S)\colon |T|=k, s\in T\} } \!\!\!\!\!\!   V^{\langle T \rangle}  \to \!\!\!\!\!\! 
\bigoplus_{\{T\in I(S)\colon |T|=k\} } \!\!\!\!\!\!  V^{\langle T \rangle}  
\to \!\!\!\!\!\! 
\bigoplus_{\{T\in I(S)\colon |T|=k, s\notin T\} }\!\!\!\!\!\!  V^{\langle T \rangle}  \to 0
\end{multline*}
where the first map is the obvious inclusion and the second the obvious quotient map.  (Note that if 
$T\in I(S\setminus B_s(1))$ and $|T|=k-1$, then $T\cup \{s\}\in I(S)$ has $k$ elements and $(V^s)^{\langle T \rangle} = V^{\langle T\cup \{s\}\rangle}$.)
\end{proof}

\section{Reflection representations}

In this section, we apply Proposition~\ref{les} to explicitly calculate the Coxeter cohomology of all the finite Coxeter groups with coefficients in their respective reflection representations.  It turns out that the results will depend only on the graph of the Coxeter group, not on its labeling.  Lemma \ref{twovertices} shows this for the Coxeter groups generated by two reflections---the argument does not depend on the angle between the two axes of reflection, so long as it is not $\pi/2$.  Once that is established, the independence from labeling is propagated to larger graphs because of the inductive nature of the calculation.

The result is as follows.

\begin{thm}
\label{main-ref}
Let $G$ be a Coxeter group of rank $n$, and $V_n$  the reflection representation of $G$.  Then
$H_C^i(G,V_n) = 0$ except in the following cases:
\begin{enumerate}
\item $H_C^i(G,V_n) = \R^{i-1}$ if $n=3i-1$ and $G\notin\{\D_n,\E_8\}$.
\item $H_C^i(G,V_n) = \R^{2i}$ if $n=3i$ and $G\notin\{\D_n,\E_6\}$.
\item $H_C^i(G,V_n) = \R^{i+1}$ if $n=3i+1$ and $G\notin\{\D_n,\E_7\}$.
\item $H_C^i(\D_{3i+3},V_{3i+3}) = \R^{i+2}$.
\item $H_C^i(\D_{3i+4},V_{3i+4}) = \R^{2i+3}$.
\item $H_C^i(\D_{3i+5},V_{3i+5}) = \R^{i+1}$.
\item $H_C^2(\E_6, V_6) = \R$.
\item $H_C^2(\E_7, V_7) = \R^2$.
\item $H_C^1(\E_8, V_8) = \R$.
\end{enumerate} 
\end{thm}

We will prove this theorem in several steps,
beginning with a result about Coxeter cohomology with trivial coefficients:

 \begin{prop}
\label{no-cycle}
For all $n\ge 1$,
\begin{equation}
\label{An}
\dim_K H_C^i(\A_n,K) = 
\begin{cases}
1&\mbox{if } n\in\{3i-1,3i\}, \\
0&\mbox{otherwise.}
\end{cases}
\end{equation}\end{prop}

\begin{proof}
This can be deduced easily from Proposition~\ref{les} and Proposition~\ref{kunneth} by induction on $|S|$.
It follows from the stronger result of Ehrenborg and Hetyei \cite{EH} that each $|I(S)|$ is either contractible or homotopy equivalent to a sphere and that it is a sphere of dimension $i$ if and only if $n\in\{3i-1,3i\}$.
\end{proof}

\begin{lemma}
\label{onevertex}
For $V_1$ the reflection representation of $A_1$,
$$
H_C^i(\A_1,V_1) =
\begin{cases}
\R&\hbox{if } $i=0$, \\
0&\mbox{otherwise.}
\end{cases}
$$
\end{lemma}
\begin{proof}
The Coxeter cochain cohomology will, in this case, be the cohomology of the complex $0\to V_1\to V_1^{\langle A_1\rangle}\to 0$, and
$V_1^{\langle A_1\rangle}=0$. 
\end{proof}

\begin{lemma}
\label{twovertices}
If $G$ is any of the finite Coxeter groups with two generators ($\A_2$, $\B_2=\CC_2$, $\GG_2$, $\H_2$, or any $\I_2( p )$) and $V_2$ its reflection representation,
$$
H_C^*(G,V_2) =0.$$
\end{lemma}
\begin{proof}
Let the generators of $G$ by $s_1$ and $s_2$, where $s_i$ reflects along the line perpendicular to the root $\alpha_i$, $i=1,2$.  The
Coxeter cochain cohomology will be the cohomology of the complex $0\to V_2\to V_2^{\langle \{s_1\} \rangle}\oplus 
V_2^{\langle \{s_2\} \rangle} \to 0$.  Since $V_2=\Span(\alpha_1, \alpha_2)$, any $v\in V_2$ can be written $v=a\alpha_1+b\alpha_2$ for some $a,b\in\R$, and then the coboundary map will send $a\alpha_1+b\alpha_2\mapsto(b(\alpha_2+s_1(\alpha_2)), a(\alpha_1+s_2(\alpha_1))).$
But  $\alpha_2+s_1(\alpha_2)$ is twice the projection of $\alpha_2$ onto $\alpha_1$, which is nonzero since $\alpha_1$ is not orthogonal to $\alpha_2$ (the angle between them is, in fact, $\pi-{\pi\over k}$ where $k\geq 3$ is the label on the edge between $s_1$ and $s_2$ in the Coxeter graph).  And similarly $\alpha_1+s_2(\alpha_1)$ is nonzero, so the middle map in the cochain complex is an isomorphism. 
\end{proof}

\begin{lemma}
\label{threevertices}
If $G$ is any of the finite Coxeter groups with three generators ($\A_3$, $\B_3=\CC_3$,  or $\H_3$) and $V_3$ its reflection representation,
$$
H_C^i(G,V_3) =
\begin{cases}
\R^2&\hbox{if  $i=1$,} \\
0&\mbox{otherwise.}
\end{cases}
$$

\end{lemma}
\begin{proof}
Let $s$ be the middle vertex in the Coxeter graph of $G$.  We want to apply Proposition \ref{les} for this $s$.  Then  $G^s$ is the trivial group, because if we remove $B_s(1)$, no generators are left.  Thus the Coxeter cochain complex corresponding to $G^s$ (and no generators) with coefficients in $V_3^s$ has only the copy of the coefficients $V_3^s$ in dimension zero, corresponding to the empty set,  and $H^*_C(G^s,V_3^s)$ consists only of a copy of $V_3^s\cong \R^2$ in dimension zero.

On the other hand, if we only remove the middle vertex $s$ itself, we are left with two commuting generators.  Thus $G_s\cong \A_1\times \A_1$ is generated by two commuting reflections, i.e., reflections through planes $\alpha_1^\perp$ and $\alpha_3^\perp$, where 
$\alpha_1\perp \alpha_3$.
So there is an orthogonal decomposition $V_3=\Span(\alpha_1)\oplus \Span(\alpha_3)\oplus C$, where $C$ is the one-dimensional orthogonal complement to the $\Span(\alpha_1,\alpha_3)$,  on which both $s_i$ act trivially since it is  orthogonal to both $\alpha_i$.

We get that
$$V_3 |_{\A_1\times \A_1} \cong V_1\boxtimes\R \oplus \R\boxtimes V_1\oplus \R\boxtimes\R,$$
and so by  Lemma \ref{kunneth}, 
\begin{equation*}
\begin{split}
 H  ^*_C(  \A_1   \times \A_1, V_3   |_{\A_1\times   \A_1} & )
\\
 \cong 
H^*_C(\A_1,V_1) 
\otimes &
H^*_C(\A_1,\R)
\oplus
H^*_C(\A_1,\R)
\otimes
H^*_C(\A_1,V_1)
\\ &\oplus
H^*_C(\A_1,\R)
\otimes
H^*_C(\A_1,\R)
 =0
\end{split}
\end{equation*}
by Proposition \ref{no-cycle}.
It therefore will contribute nothing to $H^*_C(G,V_3)$ which will,  by Proposition \ref{les}, be zero everywhere except
$ H_C^1(G,V_3)\cong H^0_C(G^s,V_3^s)\cong \R^2$.
\end{proof}

\begin{prop}
\label{AnVncoh}
Let $G$ be a finite Coxeter group with $n$ generators whose Coxeter graph consists of a single line, and let $V_n$ be its reflection representation.  Then for all $i\ge 0$,
$$
H_C^i(G,V_n) =
\begin{cases}
\R^{2i}&\hbox{if  $n=3i$,} \\
\R^{i+1}&\hbox{if $n=3i+1$,}\\
\R^{i-1}&\hbox{if  $n=3i-1$, }\\
0&\mbox{otherwise.}
\end{cases}
$$
\end{prop}

\begin{proof}
We have dealt with the cases where $n\leq 3$ in Lemmas \ref{onevertex}, \ref{twovertices}, and \ref{threevertices}, so we are left with the cases $\A_n$ ($n\ge 4$), $\B_n=\CC_n$ ($n\ge 4$), $\FF_4$ and $\H_4$.  We will write out the proof for the $\A_n$.  The proof for all the groups with four generators  is identical to the one for $\A_4$; that for the higher for $\B_n=\CC_n$ is the same except that  one should replace $\A_{n-3}$ with $\B_{n-3}$, and for $n\geq6$ one should do the same for $\A_{n-4}$.

The proof is by induction on $n\geq 4$.  We will use Proposition \ref{les}, with $s=s_{n-2}$, 
corresponding to the third vertex from the right.   If each $s_i$, $1\leq i\leq n$ is a reflection along the hyperplane perpendicular to $\alpha_i$, then as in the proof of Lemma \ref{threevertices}, vectors $\alpha_i,\alpha_j$ which correspond to distinct commuting $s_i,s_j$ are orthogonal to each other.
We have
\begin{equation*}
\langle S \setminus\{s\}\rangle =\A_{n-3}\times  \A_2, \quad \langle S \setminus B_s(1)\rangle =\A_{n-4} \times   \A_1,
\end{equation*}
and we have an orthogonal decomposition
\begin{equation}
\label{threeparts}
V_n\cong \Span(\alpha_1,\ldots, \alpha_{n-3}) \oplus\Span(\alpha_{n-1},\alpha_n)\oplus C,
\end{equation}
where $C$ is the one-dimensional orthogonal complement of the other two summands, so
$$V_n|_{\A_{n-3}\times \A_2} = V_{n-3}\boxtimes \R \oplus \R\boxtimes V_2\oplus \R\boxtimes \R.$$
By Lemma \ref{kunneth}, 
\begin{equation*}
\begin{split}
& H^*_C({\A_{n-3}\times \A_2},V_n|_{\A_{n-3}\times \A_2})
\\
& \cong
H^*_C(\A_{n-3},V_{n-3})\otimes
H^*_C(\A_2, \R)
\\
&\quad \oplus
H^*_C(\A_{n-3},\R)\otimes
H^*_C(\A_2, V_2)
\\
&\quad \oplus
H^*_C(\A_{n-3},\R)\otimes
H^*_C(\A_2, \R).
\end{split}
\end{equation*}
Applying Theorem \ref{no-cycle} to $H^*_C(\A_k, \R)$,  and applying the inductive hypothesis to $H^*_C(\A_k, V_k)$, we get 
\begin{equation}
\label{wos}
\begin{split}
 H^i_C( & \A_{n-3}\times  \A_2, V_n|_{\A_{n-3}\times \A_2})
 \\
 &=
 \begin{cases}
\R^{2(i-1)}\oplus 0\oplus\R=\R^{2i-1}&\hbox{if } n=3i, \\
\R^{i}\oplus 0\oplus 0=\R^i&\hbox{if }n=3i+1,\\
\R^{i-2}\oplus 0\oplus\R=\R^{i-1}&\hbox{if } n=3i-1, \\
0&\mbox{otherwise.}
\end{cases}
\end{split}
\end{equation}
Similarly, we can decompose
\begin{equation*}
\label{threeprimeparts}
V_n\cong \Span(\alpha_1,\ldots, \alpha_{n-4}) \oplus\Span(\alpha_n)\oplus C,
\end{equation*}
where $C$ is the orthogonal complement of the other two summands, and this time the dimension of $C$ is three.  Now the first two summands are perpendicular to $\alpha_{n-2}$, so $s_{n-2}$ acts nontrivially only on the third summand and $V_n^{s_{n-2}}$ has a splitting like
(\ref{threeprimeparts}) but with a two-dimensional $C$, and 

$$V^s_n|_{\A_{n-4}\times \A_1} = V_{n-4}\boxtimes \R \oplus \R\boxtimes V_1\oplus ( \R\boxtimes \R)^{\oplus 2}.$$
By Lemma \ref{kunneth} (noting that $H^*_C(\A_1,\R)=0$ by Theorem \ref{no-cycle}), we get
\begin{equation*}
\begin{split}
& H^*_C(\A_{n-4}\times \A_1,V_n|_{\A_{n-4}\times \A_1})
\\
& \cong
H^*_C(\A_{n-4},V_{n-4})\otimes
H^*_C(\A_1, \R)
\\
&\quad \oplus
H^*_C(\A_{n-4},\R)\otimes
H^*_C(\A_1, V_1)
\\
&\quad \oplus
\bigl(H^*_C(\A_{n-4},\R)\otimes
H^*_C(\A_1, \R)\bigr)^{\oplus 2}
\\
&\cong
H^*_C(W(A_{n-4}),\R)\otimes
H^*_C(\A_1, V_1)\cong H^*_C(W(A_{n-4}),\R) ,
\end{split}
\end{equation*}
which by  Theorem \ref{no-cycle} gives
\begin{equation}
\label{woball}
H^i_C(\A_{n-4}\times \A_1,V_n|_{\A_{n-4}\times \A_1})=
 \begin{cases}
\R&\hbox{if } n=3i+3, \\
\R&\hbox{if }n=3i+4,\\
0&\mbox{otherwise.}
\end{cases}
\end{equation}
Putting together the information in equations (\ref{wos}) and (\ref{woball}), we get that for $G=\A_n$, $V=V_n$, and $s=s_{n-2}$, the long exact sequence of Proposition \ref{les} vanishes everywhere except the segment
$$\to H_C^{i-1}(\bigl(\A_n\bigr)^s,V_n ^s)\to H_C^i(\A_n,V_n )\to H_C^i(\bigl(\A_n\bigr)_s,V_n )\to$$
where $n\in \{3i, 3i+1, 3i-1\}$, meaning that for these $n$,
\begin{equation*}
\begin{split}
 H_C^i(\A_n,V_n )
 & \cong  H_C^{i-1}\bigl(\bigl(\A_n\bigr)^s,V_n ^s\bigr)\oplus H_C^i(\bigl(\A_n\bigr)_s,V_n\bigr)
 \\
 & \cong
  \begin{cases}
\R\oplus \R^{2i-1}&\hbox{if } n=3i, \\
\R\oplus \R^i&\hbox{if }n=3i+1,\\
0\oplus \R^{i-1}&\hbox{if }n=3i-1,\\
0&\mbox{otherwise,}
\end{cases}
\end{split}
\end{equation*}
which is exactly what we needed to show.
\end{proof}

\begin{prop}
\label{DnVncoh}
Let  $n\geq 4$, and let $V_n$ be the reflection representation of $\D_n$.  Then for all $i\ge 0$,
$$
H_C^i(\D_n,V_n) =
\begin{cases}
\R^{i+2}&\hbox{if  $n=3i+3$,} \\
\R^{2i+3}&\hbox{if $n=3i+4$,}\\
\R^{i+1}&\hbox{if  $n=3i+5$,} \\
0&\mbox{otherwise.}
\end{cases}
$$
\end{prop}

\begin{proof}
We will use Proposition \ref{les}   with $s=s_{n-2}$, the trivalent vertex in the Coxeter graph.
We have
\begin{equation*}
\langle S \setminus\{s\}\rangle =\A_{n-3}\times  \A_1\times \A_1, \quad \langle S \setminus B_s(1)\rangle =\A_{n-4}.
\end{equation*}
Since the roots corresponding to distinct commuting generators are orthogonal to each other, we have an orthogonal decomposition
\begin{equation*}
V_n\cong \Span(\alpha_1,\ldots ,\alpha_{n-3}) \oplus\Span(\alpha_{n-1})\oplus\Span(\alpha_n)\oplus C,
\end{equation*}
where $C$ is the one-dimensional orthogonal complement of the other three summands, so
\begin{equation}
\label{Dsplit}
V_n|_{\A_{n-3}\times \A_1\times A_1} = V_{n-3}\boxtimes \R\boxtimes \R \oplus \R\boxtimes V_1 \boxtimes \R
\oplus \R\boxtimes \R \boxtimes V_1
\oplus \R\boxtimes \R\boxtimes \R.
\end{equation}
By Lemma \ref{kunneth}, however, since $H^*_C(\A_1;\R)=0$ (by Theorem \ref{no-cycle})
and each summand in (\ref{Dsplit})
 has at least one factor which is the trivial representation of $\A_1$,
$H^*_C({\A_{n-3}\times \A_1\times \A_1},V_n|_{\A_{n-3}\times \A_1\times \A_1})=0$.

By Proposition \ref{les},  $H^*_C(\D_n, V_n)\cong H^*_C(\A_{n-4}, V_n^{s_{n-2}}|_{A_{n-4}})$.
There is an orthogonal decomposition
\begin{equation*}
V_n\cong \Span(\alpha_1,\ldots, \alpha_{n-4}) \oplus C,
\end{equation*}
where $C$ is four-dimensional, and $s_{n-2}$ acts trivially on the first summand and nontrivially on $C$.
 Then
 $V^s_n|_{\A_{n-4}} = V_{n-4}\oplus\R^3$, 
 and by Theorem \ref{AnVncoh} for the first summand and Theorem \ref{no-cycle} for the second, we get
\begin{equation*}
H^i_C(\D_n, V_n)\cong H^i_C(\A_{n-4}, V_n^{s_{n-2}}|_{A_{n-4}})
\cong
 \begin{cases}
\R^{2i} \oplus\R^3&\hbox{if  $n-4=3i$,} \\
\R^{i+1} \oplus 0&\hbox{if $n-4=3i+1$,}\\
\R^{i-1} \oplus \R^3&\hbox{if $n-4=3i-1$,}\\
0&\mbox{otherwise.}
\end{cases}
\end{equation*}
which is exactly what we needed to show.
\end{proof}

\begin{prop}
\label{EnVncoh}
Let   $V_n$ be the reflection representation of $\E_n$, for $n=6, 7,8$.  Then   $H^2_C(\E_6, V_6)\cong \R$, $H^2_C(\E_7, V_7)\cong \R^2$,
$H^1_C(\E_8, V_8)\cong \R$, and all the rest of the Coxeter cochain cohomology groups of the $\E_n$ vanish.
\end{prop}

\begin{proof}
We will use Proposition \ref{les}, with $s=s_{4}$, the trivalent vertex in the Coxeter graph.
We have
\begin{equation*}
\langle S \setminus\{s\}\rangle =\A_{2}\times  \A_1\times \A_{n-4}, \quad \langle S \setminus B_s(1)\rangle =\A_1\times \A_{n-5}.
\end{equation*}
Since the roots corresponding to distinct commuting generators are orthogonal to each other, we have an orthogonal decomposition
\begin{equation*}
V_n\cong \Span(\alpha_1, \alpha_{3}) \oplus\Span(\alpha_{2})\oplus\Span(\alpha_5,\ldots,\alpha_n)\oplus C,
\end{equation*}
where $C$ is the one-dimensional orthogonal complement of the other three summands, so
\begin{equation}
\label{Esplit}
V_n|_{\A_{2}\times  \A_1\times \A_{n-4}} =
 V_{2}\boxtimes \R\boxtimes \R \oplus \R\boxtimes V_1 \boxtimes \R
\oplus \R\boxtimes \R \boxtimes V_{n-4}
\oplus \R\boxtimes \R\boxtimes \R.
\end{equation}
We use Lemma \ref{kunneth}, remembering that $H^*_C(\A_1;\R)=0$ (by Theorem \ref{no-cycle}).
Only one summand in 
(\ref{Esplit}) does not have the trivial representation of $\A_1$ as a factor, so
$$H^*_C(\A_{2}\times  \A_1\times \A_{n-4},V_n|_{\A_{2}\times  \A_1\times \A_{n-4}}) \cong H^*_C(\A_2,\R)\otimes H^*_C(\A_1,V_1) \otimes H^*_C(\A_{n-4},\R).$$
Using Theorem \ref{no-cycle} and Theorem \ref{AnVncoh}, the first factor is $\R$ in dimension 1, the second is $\R$ in dimension $0$, and the third one is $\R$ in dimension $1$ for $n=6,7$ but $0$ for $n=8$.  So in total we get
$$H^i_C(\A_{2}\times  \A_1\times \A_{n-4},V_n|_{\A_{2}\times  \A_1\times \A_{n-4}})
\cong
 \begin{cases}
\R&\hbox{if } i=2\hbox{ and } n=6\hbox{ or } 7, \\
0&\mbox{otherwise.}
\end{cases}
$$
We can also write the orthogonal decomposition
\begin{equation*}
V_n\cong \Span(\alpha_1)\oplus \Span(\alpha_6,\ldots, \alpha_{n}) \oplus C,
\end{equation*}
where $C$ is four-dimensional, and $s_{4}$ acts trivially on the first summand and nontrivially on $C$.
 Then
 \begin{equation*}
\label{Eprimesplit}
V_n^{s_{4}}|_{\A_1\times \A_{n-5}} =
 V_{1}\boxtimes \R
  \oplus \R\boxtimes V_{n-5}
\oplus( \R\boxtimes \R)^{\oplus 3},
\end{equation*}
and using Lemma \ref{kunneth} and $H^*_C(\A_1;\R)=0$,
$$H^*_C(\A_1\times \A_{n-5}, V_n^{s_{4}}|_{\A_1\times \A_{n-5}} )\cong H^*_C(\A_1,V_1)
\otimes H^*_C( \A_{n-5}, \R).$$
The first factor is $\R$ in dimension $0$, and the second is zero if $n=6$, but $\R$ in dimension $1$ if $n=7$ or $8$.
Thus
$$H^*_C(\A_1\times \A_{n-5}, V_n^{s_{4}}|_{\A_1\times \A_{n-5}} )\cong
 \begin{cases}
\R&\hbox{if } i=1\hbox{ and } n=7\hbox{ or } 8, \\
0&\mbox{otherwise.}
\end{cases}
$$
and the result follows by Proposition \ref{les}.
\end{proof}

\section{Arrangements and fatpoints}

In this section, we explain how the relation between the homology of the space $\X_{n,3}$ and
$\Tor^{A_{m,3}}(\C,\C)$  which was described in \cite{PRW} can be understood in terms of Coxeter cohomology for groups of type $\A$.

We begin on the configuration space side.  
Let $\Delta_{n,3}$ denote the closed $\A_{n-1}$-stable
subset of $[0,1]^n$ consisting of $n$-tuples for which some value in $[0,1]$ 
appears at least three times as a coordinate.
Following \cite{PRW}, we use Alexander duality to
identify $H_k(\X_{n,3};\C)$  as a representation of the group $\A_{n-1}$  with 
$H_{n-k}([0,1]^n, \partial[0,1]^n\cup \Delta_{n,3};\C)
\otimes \sgn $,
the sign representation appearing as the action of $\A_{n-1}$ on the orientations of $([0,1]^n,\partial[0,1]^n) $.

We consider the following simplicial decomposition of $[0,1]^n$: the vertices of the cube are the vertices, ordered lexicographically.  A $k$-simplex is determined by a partition of the index set $\{1,2,\ldots, n\}$ into an ordered $(k+2)$-tuple of subsets $(G_0,G_1,\ldots,G_k,G_{k+1})$ with $G_0$ and $G_{k+1}$ possibly empty but $G_i\neq \emptyset$ for $1\leq i\leq k$, and consists of 
\begin{equation*}
\begin{split}
\{(x_1,x_2,\ldots x_n)\in X^n:\ & {\mathrm {there\ exist} \ }0=t_0\leq t_1\leq\cdots\leq t_k\leq t_{k+1}=1\  \cr &
{\mathrm {so \ that} }\ x_j=t_i\ \forall j\in G_i\}.
\end{split}
\end{equation*}
Clearly, $\Delta_{n,3}$ is a subcomplex with respect to this structure.
The natural basis of the quotient complex of the chain complex of $[0,1]^n$ by the subcomplex associated to $\partial[0,1]^n\cup \Delta_{n,3}$ is indexed by ordered partitions 
$(G_1,G_2,\ldots,G_k)$ of $\{1,2,\ldots, n\}$, where $|G_i|\in\{1,2\}$ for all $1\leq i \leq k $.  For such an ordered partition, every coordinate of the vector $(|G_1|,|G_2|,\ldots,|G_k|)$, is $1$ (occurring $2k-n$ times) or $2$ (occurring $n-k$ times).

There is a bijection between the vectors $(|G_1|,|G_2|,\ldots,|G_k|)$, with $n-k$ coordinates equal to $2$ and the rest equal to $1$, and independent subsets $T\in I(S)$ with $|T|=n-k$, namely: let $s_j\in T$ if and only if there exists $i$ such that $j=|G_1|+|G_2|+\cdots + |G_{i-1}|$ and $|G_{i}|=2$.

There is a transitive right $S_n$ action on all ordered partitions 
$(G_1,G_2,\ldots,G_k)$ of $\{1,2,\ldots, n\}$ with a fixed vector of lengths $(|G_1|,|G_2|,\ldots,|G_k|)$ via 
$$(G_1,G_2,\ldots,G_k)\sigma =(\sigma^{-1}(G_1), \sigma^{-1}(G_2),\ldots,\sigma^{-1}(G_k)).$$
If we assume that $(G_1,G_2,\ldots,G_k)$  is an ordered partition so that $j_1\in G_{i_1}, j_2\in G_{i_2}$, 
 and $i_1<i_2$ implies $j_1< j_2$, that is: a partition which preserves the order of $\{1,2,\ldots, n\}$, then the stabilizer of $(G_1,G_2,\ldots,G_k)$ is exactly $\langle T \rangle$, where $T$ is the subset of $S$ corresponding to $(|G_1|,|G_2|,\ldots,|G_k|)$.
 
 The group $S_n$ also acts from the right on $K[S_n]^{\langle T\rangle}$, where the invariants are taken with respect to the left action.  If we assume as before that $(G_1,G_2,\ldots,G_k)$ is an ordered  partition which preserves the order of $\{1,2,\ldots, n\}$, then the stabilizer of $\prod_{s_j\in T}(\id + s_j)\in K[S_n]^{\langle T\rangle}$ is again $\langle T\rangle$, and the orbit of $\prod_{s_j\in T}(\id + s_j)$ is all of $K[S_n]^{\langle T\rangle}$ (being the image of the $\langle T\rangle$-trace).  Using all this we define our isomorphism
$$\phi:\ C_k([0,1]^n, \bigl(\partial [0,1]^n\cup\Delta_{n,3}\bigr); K) \to X^{n-k}_C$$
as follows:

Given an ordered partition $(G_1,G_2,\ldots,G_k)$ which preserves the order of $\{1,2,\ldots, n\}$ as above,
$$\phi \bigl( (G_1,G_2,\ldots,G_k)\bigr)= (-1)^{\sum_{j=1}^{k-1} (|G_{j+1}|+\cdots + |G_k|)}
\prod_{s_j\in T}(\id + s_j)$$
in $K[S_n]^{\langle T\rangle}$, where $T$ is the independent subset of $S$ which corresponds to  $(|G_1|,|G_2|,\ldots,|G_k|)$ and therefore consists of $n-k$ elements.  It is easily seen that this induces an isomorphism of chain complexes.
We deduce the following theorem:

\begin{thm}
\label{complementvsCoxeter}
There is an isomorphism of right $S_n$-representations
\begin{equation*}
\begin{split}
H_k(X_{n,3}; \C)\cong H^{n-k}_C(\A_{n-1}, \C[S_n])\otimes\sgn 
\end{split}
\end{equation*}
for all $0\leq k \leq n$, where $\C[S_n]$ is viewed as a $\A_{n-1}\cong S_n$-module with respect to left multiplication.   
\end{thm}

The decomposition 
$$\C[S_n] = \bigoplus_V V\boxtimes V^*$$ 
and the self-duality of every irreducible representation of
$S_n$ shows that the multiplicity of the $V$-isotypic component of $H_k(X_{n,3}; \C)$ is $H^{n-k}_C(\A_{n-1},V\otimes\sgn)$.

On the Tor side, we follow \cite{PRW}, using the minimal resolution of $K$ as $A$-module to compute 
$\Tor_i^{A_{n,3}}(\C,\C)$ as the $i$th homology of the complex
$$\cdots
 \M\otimes\M \otimes \M\stackrel{-d_1+d_2}{\to}
\M \otimes \M\stackrel{-d_1}{\to}
\M\stackrel{0}{\to} \C,$$
where $\M$ is the maximal ideal of $A$.  As $\GL_n(\C)$-representation, $\M\cong V\oplus \Sym^2 V$,
where $V$ is the standard representation of $\GL_n(\C)$.
Thus
\begin{equation}
\label{m-sum}
\M^{\otimes i} \cong
 ( V\oplus S^2 V)^{\otimes i}
 \cong\bigoplus_{0\leq j\leq i}
 \ \bigoplus_{\substack{\Sigma\subset\{1,\ldots, i\} \\ |\Sigma| = j}}
 (V^{\otimes(i+j)} ) ^{\langle T_\Sigma \rangle},
 \end{equation}
 where the permutation group $\A_{i+j-1}\cong S_{i+j}$ permutes the $i+j$ tensored factors in $V^{\otimes(i+j)}$, and $\langle T_\Sigma \rangle$ is the subgroup of  $\A_{i+j-1}$ 
 generated by the set of transpositions 
 $$\{s_k\colon k=h+|\{1,2,\ldots,h-1\}\cap \Sigma|, h\in \Sigma\}.$$

We define a map
$$\phi:\ 
 \bigoplus_{0\leq j\leq i }
 X^j_C(\A_{i+j-1}, V^{\otimes(i+j)})\stackrel{\cong}\to \M^{\otimes i}$$
as follows:
If $0\le k_1 - 1 < k_2 - 2 < \cdots < k_j-j < i$, i.e., if $T=\{s_{k_1},s_{k_2},\ldots, s_{k_j}\}$ is
a set of disjoint transpositions in $S_{i+j}$, and if $v\in ( V^{\otimes(i+j)})^{\langle T \rangle}$, we set
\begin{equation*}
\begin{split}
\phi  (v)= & (-1)^{\Sigma_{a=1}^j k_a} {\frac 1 {2^j}}[v] 
\\
&\in V^{\otimes(k_1-1)}
 \otimes S^2 V
 \otimes
  V^{\otimes(k_2-k_1-2)}
 \otimes S^2 V
 \otimes
\\
& \quad \quad\cdots 
  \otimes
  V^{\otimes(k_j-k_{j-1}-2)}
 \otimes S^2 V
 \otimes
 V^{\otimes(i+j-k_j-1)}.
 \end{split}
 \end{equation*}

It is easy to check that this defines an isomorphism of chain complexes, and one concludes as follows:

\begin{theorem}
\label{Torcalc}
For $i\geq 1$,
$$\Tor_i^{A_{m,3}}(\C,\C)\cong \bigoplus_{0\leq j\leq i} H^j_C(\A_{i+j-1} , V^{\otimes (i+j)})$$
as representations of $\GL_m(\C)$, where $V$ is the standard representation of $\GL_m(\C)$ on $\C^m$.
\end{theorem}

By means of Schur-Weyl duality, the multiplicity of any $\GL_m(\C)$-representation in $\Tor_i^{A_{m,3}}(\C,\C)$ can therefore be expressed in terms of Coxeter cohomology.

\end{document}